\documentclass[11pt]{amsart}
\usepackage{tikz}
\usepackage{color}

\begin{document}

\title[Distinguishing index]{A class of graphs with distinguishing index $\bf D' \leq 3$}
\author{Mariusz Grech, Andrzej Kisielewicz}
\address{Faculty of Pure and Applied Mathematics, Wrocław University of Science and Technology \\
Wybrzeże Wyspiańskiego Str. 27,
50-370 Wrocław, Poland}
\email{[mariusz.grech,andrzej.kisielewicz]@pwr.edu.pl}
\thanks{
{Supported in part by Polish NCN grant 2016/21/B/ST1/03079}}

\newtheorem{Theorem}{Theorem}
\newtheorem{Lemma}{Lemma}[section]

\newcommand{\ppp}{\colorbox{yellow}{???}} 

\newcommand{\rrr}{\colorbox{red}{===}}

\begin{abstract}

An edge-coloring of a graph is called asymmetric if the only automorphism which preserves it is the identity.  Lehner, Pil\'{s}niak, and Stawiski proved that all connected regular graphs except $K_2$ admit an asymmetric edge-coloring with three colors. We generalize this result for graphs whose minimal degree $\delta$ and the maximal degree $\Delta$ satisfy $\delta \geq \Delta/2$. 
\end{abstract}
 
\keywords{asymmetric coloring, edge-coloring,  distinguishing index, automorphism group of graph}

\maketitle

Let $G$ be a connected, finite or infinite, graph. 
An edge-coloring  $\phi : E(G) \to \{1,2,...,r\}$ of a graph $G$ is said to be \emph{asymmetric} if no nontrivial automorphism of $G$ preserves colors of he edges. The point is to destroy the symmetries of the graph, that is, to make the automorphism group of the colored graph trivial. Of course we are interested in the minimal number $r$ of colors for which an asymmetric coloring of $G$ exists. Such a number is called the \emph{distinguishing index} of $G$ and denoted by $D'(G)$. Note that, in this definition, $\phi$ above is an arbitrary function from $E(G)$ to $\{1,2,...,r\}$, with no assumption that adjacent vertices get different colors. The notion of the distinguishing index has been introduced in \cite{KP}. This was done in analogy to the concept of the distinguishing number, referring to coloring vertices rather than edges, which has been introduced and considered much earlier \cite{AC,ba}.

The very first result on $D'(G)$, obtained together with introducing the distinguishing index in \cite{KP}, was
that for connected graphs with finite maximum degree $\Delta$,  $D'(G) \leq \Delta$, unless $G$ is $C_3, C_4$ or $C_5$. This was improved in \cite{pi,PS} by characterizing graphs for which the equality holds. There are many classes of graphs for which much better bounds are possible. It is known that, apart from finitely many exceptions, $D'(G) \leq 2$ for all traceable graphs \cite{pi}, 3-connected planar graphs \cite{PT}, Cartesian powers of finite and countable graphs \cite{BP,FI},
and countable graphs where every non-trivial automorphism moves infinitely many edges \cite{le}.
In turn, for line graphs and claw-free graphs $D'(G) \leq 3$ \cite{AS}. A Nordhaus-Gaddum type inequality $D'(G)+D'(\overline{G}) \leq \Delta + 2$, under some natural conditions, has been proved in \cite{pi2}.
For connected graphs without pendant edges the bound has been improved to $\sqrt{\Delta} +1$ \cite{IKP}.
Let us mention also that in \cite{GK}, using an approach of permutation group theory, the equality $D'(G)=2$ has been proved for all graphs whose automorphism group is simple.
Finally, in \cite{LPS}, Lehner, Pil\'{s}niak, and Stawiski  proved that for regular graphs other than $K_2$, $D'(G)\leq 3$. We refer the reader to \cite{LPS} for more on motivation and other related results.

Unfortunately, the proof in the latter paper contains a gap. The inequality in line~9 on page~6 of \cite{LPS} is false for $f=0$. A correction seems two require considering at least two additional subcases. Rather than correcting the gap, we provide a new proof that yields the same bound for a larger class of graphs whose the minimum and the maximum degree are not too far from each other.

\begin{Theorem}
Let $G$ be a connected graph with the finite maximal degree $\Delta$ and minimal degree $\delta$. If $\delta \geq \Delta/2$  and $G\neq K_2$, then $G$ admits an asymmetric edge-coloring with three colors.
\end{Theorem}

We note that the graph $G$ in the theorem may be finite or infinite. As observed in \cite{LPS}, there are few small graphs satisfying the conditions of the theorem for which $D'(G) =3$. These are regular graphs $K_{n}$ for $n\leq 5$, $K_{n,n}$ for $n\leq 3$, $C_5$, and one non-regular graph $K_{2,4}$. 

\section{Preliminaries}\label{s:pre}

It is an easy exercise to check that for each complete graph $K_n$, $n\geq 3$, there exists an asymmetric $3$-coloring satisfying an additional condition that each vertex has an incident edge whose color is other than $3$. In fact, for $n\geq 6$, $K_n$ has an asymmetric $2$-coloring, easy to construct, so it remains to check  $K_n$ for $n=3,4,5$. In the proof we will refer to this fact without further comment.

We will consider the subgroups of the automorphism group $Aut(G)$ of a graph $G$. The stabilizer of a vertex $x_0$, denoted $Aut(G,x_0)$ consists of the automorphism of $G$ fixing $x_0$. For this and further subgroups $\Gamma\subseteq Aut(G)$ we consider the orbits of $\Gamma$, that is, the subsets $O\subseteq V(G)$ of those vertices of $G$ that can be moved into each other by an automorphism in $\Gamma$. (An orbit of a vertex $z\in V(G)$ is the set $\{x\in V(G): \sigma(z)=x$ for some $\sigma\in\Gamma\}$). Orbits of $Aut(G,x_0)$ has been considered in the procedure of coloring edges of a graph $G$ in \cite{LPS}. Unlike in \cite{LPS} we consider further subgroups and new orbits at each step of our procedure of coloring edges of $G$.

After coloring some edges of $G$, we will consider the subgroup of $Aut(G,x_0)$ that preserves the colors of the edges colored so far. It consists of those automorphism of $G$, fixing $x_0$, that satisfy additionally the condition that the edge $\sigma(x)\sigma(y)$ has the same color as the edge $xy$, for all colored edges $xy\in E(G)$. In any case, we will take care that the edges colored so far are mapped into colored edges under any automorphism in question. Formally, the edges colored so far will form, at any step, a union of orbitals of the currently considered subgroup of $Aut(G)$.

The subgroups of $Aut(G)$ considered at successive steps will form a descending chain of sets. Accordingly, the sets of orbits considered at successive steps will form a descending chain of partitions of $V(G)$, in the sense that consecutive partitions are finer and finer. The goal is to get the smallest partition into singletons after coloring all the edges.

In the proof below we make use of the fact that in any orbit $O$ (of any subgroup of $Aut(G)$)) all the vertices have the same number $k$ of outcoming edges.
One of our tools will be coloring these edges in three colors, red, blue, and green, so that the corresponding multisets of colors assigned to different vertices are different. The multisets of colors in question will be denoted $(a,b,c)$, where $a,b,c\geq 0$ are the numbers of colors red, blue, and green, respectively, and $a+b+c=k$. 
Such a multiset of colors will be called a \emph{uniform $k$-palette}, if it satisfies an additional condition: $a,b,c \leq \lceil k/2 \rceil$.

It is easy to observe the following:

\begin{Lemma}\label{l:k+1}
For each $k>0$, the number
of uniform $k$-palettes is at least $k+1$.  In addition, one may assume that each of them  has $a\leq k/2$. 
\end{Lemma}

\begin{proof}
For even $k=2m$, the additional condition is void. The $k$-palettes satisfying the conditions of the lemma are: $(i,m-j,m-i+j)$, where $i=0,\ldots,m$, $j=0,\ldots,i$. The number of such palettes is $\sum_{i=0}^{m} (i+1) = m^2/2 + 3m/2 +1 \geq 2m+1$, as required. 

For odd $k=2m+1$, the additional conditions is $a\leq m$. The $k$-palettes satisfying all conditions are: $(i,m+1-j,m-i+j)$, where $i=0,\ldots,m$, $j=0,\ldots,i+1$. The number of such palettes is $\sum_{i=0}^{m} (i+2) = m^2/2 + 5m/2 +2 \geq 2m+2$, as required. 

\end{proof}

We will also need an observation concerning partitioning palettes into smaller uniform palettes.

\begin{Lemma}\label{l:part}
Let $k=k_1+k_2\ldots+k_s$ with $s\geq 2$ and all $k_i>0$. Then every uniform $k$-palette can be partitioned into $s$ uniform palettes of cardinalities $k_1,k_2,\ldots,k_s$, respectively.
\end{Lemma}

It is enough to observe that this can be done for $s=2$ with partitioning elements as uniformly as possible. The details are left to the reader as an exercise. (We note that this partitioning may not preserve the condition $a\leq k/2$.)

In the proof below we will need sequences of uniform $k$-palettes rather than single $k$-palettes. Let 
$k=k_1+k_2\ldots+k_s$, $s\geq 1$, by a uniform $(k_1,k_2,\ldots, k_s)$-palette we mean a sequence of uniform $k_i$-palettes with $i=1,2,\ldots,s$. Two such $(k_1,k_2,\ldots, k_s)$-palettes are different if they differ as sequences. For $s=1$, a $(k_1)$-palette is identified with the single $k_1$-palette.

\begin{Lemma}\label{l:k+1(2)}
Let $k=k_1+k_2+\ldots+k_s$, $k,s>0$ and all $k_i>0$. Then the number 
of uniform $(k_1,k_2,\ldots, k_s)$-palettes is at least $k+1$. 
\end{Lemma}

\begin{proof}
For $s=1$, Lemma~\ref{l:k+1} applies. For $s>1$, assume that $k_1$ is the largest number among all $k_i$. Then, again by Lemma~\ref{l:k+1},
the number of $k_1$-palettes is at least $k_1+1$, and for each $i>1$ the number of $k_i$-palettes is at least $2$. This yields at least $(k_1+1)\cdot 2^{s-1} > s\cdot k_1 \geq k_1+k_2+\ldots+k_s=k$ uniform $(k_1,k_2,\ldots,k_s)$-palettes, as required. 
\end{proof}

\section{Proof}

 We describe the procedure of coloring edges of $G$ in red, blue and green so that the resulting coloring is asymmetric. First we establish notation and basic facts.

\subsection{Level-orbit structure} We make use of the level structure of $G$. We choose and fix a root vertex $x_0\in G$ with the smallest degree $\delta$. The remaining vertices are ordered in accordance to the distance from $x_0$ and the orbits of the stabilizer $Aut(G,x_0)$ of $x_0$. 
First, the vertices of $G$ are partitioned into sets consisting of the vertices that are equally distant from~$x_0$. Such sets are called \emph{levels} and numbered by the distance from~$x_0$. Note that each orbit of the stabilizer $Aut(G,x_0)$ is contained entirely in one of the levels, and each level is partitioned into orbits of $Aut(G,x_0)$. Of course, it may happen that a whole level forms a single orbit. It may also happen that some orbits are trivial (consisting of one vertex), which means that the vertex in question is fixed by $Aut(G,x_0)$.

Note that each edge joins either the two vertices at some consecutive levels $i$ and $i+1$, for some $i$, or it joins the two vertices at the same level. In the latter case the two vertices may belong to the same orbit or to two different orbits at the same level. This is pictured schematically in Figure~1 (note that the diagram does not show possible inner edges within orbits and edges between orbits).

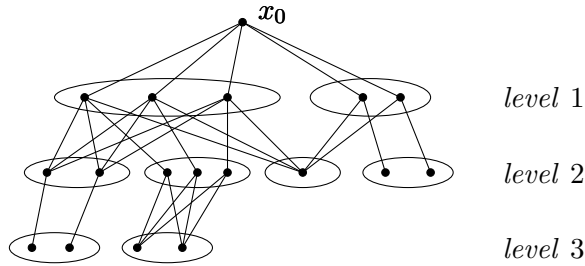
\begin{figure}\label{fig1}
\begin{tikzpicture}[auto,inner sep=1pt, 
minimum size=2pt]

  
 \draw (4,3.5) node[circle,fill=black,draw]{};   
 \draw (4.4,3.6) node {$x_0$};
  \draw (8,2.5) node {\textit{level} 1};
  
   \draw (4.4,3.6) node {$x_0$};
  \draw (8,1.5) node {\textit{level} 2};
  
   \draw (4.4,3.6) node {$x_0$};
  \draw (8,0.5) node {\textit{level} 3};
            
\draw  (3,2.5) ellipse (15mm and 2.5mm); 

\draw  (5.7,2.5) ellipse (8mm and 2.5mm); 

\draw  (1.8,1.5) ellipse (7mm and 2mm); 
\draw  (3.4,1.5) ellipse (7mm and 2mm); 
\draw  (4.8,1.5) ellipse (5mm and 2mm); 
\draw  (6.2,1.5) ellipse (6mm and 2mm); 

\draw  (1.5,0.5) ellipse (6mm and 2mm); 
\draw  (3,0.5) ellipse (6mm and 2mm);

\draw[-] (4,3.5)--(1.9,2.5);
\draw[-] (4,3.5)--(2.8,2.5);
\draw[-] (4,3.5)--(3.8,2.5);
\draw[-] (4,3.5)--(5.6,2.5);
\draw[-] (4,3.5)--(6.1,2.5);

 \draw (1.9,2.5) node[circle,fill=black,draw]{};
  \draw (2.8,2.5) node[circle,fill=black,draw]{};
   \draw (3.8,2.5) node[circle,fill=black,draw]{};
 \draw (5.6,2.5) node[circle,fill=black,draw]{};
  \draw (6.1,2.5) node[circle,fill=black,draw]{};



\draw (1.4,1.5)
node[circle,fill=black,draw]{};
\draw (2.1,1.5)
node[circle,fill=black,draw]{};

\draw (3,1.5)
node[circle,fill=black,draw]{};
\draw (3.4,1.5)
node[circle,fill=black,draw]{};
\draw (3.8,1.5)
node[circle,fill=black,draw]{};

\draw (4.8,1.5) node[circle,fill=black,draw]{};
\draw (5.9,1.5)
node[circle,fill=black,draw]{};
\draw (6.5,1.5) node[circle,fill=black,draw]{};
   
\draw[-] (1.9,2.5)--(1.4,1.5);
\draw[-] (1.9,2.5)--(2.1,1.5);
\draw[-] (1.9,2.5)--(3,1.5);
\draw[-] (1.9,2.5)--(4.8,1.5);

\draw[-] (2.8,2.5)--(1.4,1.5);
\draw[-] (2.8,2.5)--(2.1,1.5);
\draw[-] (2.8,2.5)--(3.4,1.5);
\draw[-] (2.8,2.5)--(4.8,1.5);

\draw[-] (3.8,2.5)--(1.4,1.5);
\draw[-] (3.8,2.5)--(2.1,1.5);
\draw[-] (3.8,2.5)--(3.8,1.5);
\draw[-] (3.8,2.5)--(4.8,1.5);

    
 
\draw[-] (5.6,2.5)--(4.8,1.5);
\draw[-] (6.1,2.5)--(4.8,1.5);
\draw[-] (5.6,2.5)--(5.9,1.5);
\draw[-] (6.1,2.5)--(6.5,1.5);


 

\draw (1.2,0.5)
node[circle,fill=black,draw]{};
\draw (1.7,0.5) node[circle,fill=black,draw]{};
\draw (2.6,0.5)
node[circle,fill=black,draw]{};
\draw (3.2,0.5) node[circle,fill=black,draw]{};

\draw[-] (1.4,1.5)--(1.2,0.5);
\draw[-] (2.1,1.5)--(1.7,0.5);

\draw[-] (3,1.5)--(2.6,0.5);
\draw[-] (3,1.5)--(3.2,0.5);

\draw[-] (3.4,1.5)--(2.6,0.5);
\draw[-] (3.4,1.5)--(3.2,0.5);
\draw[-] (3.8,1.5)--(2.6,0.5);
\draw[-] (3.8,1.5)--(3.2,0.5);


  \end{tikzpicture}
\caption{Level-orbit structure of $G$.  }
\end{figure}

\subsection{Initial step and outline of the procedure.} \label{s:in}
As the initial step of our procedure,  we color red all the edges outcoming from $x_0$. As we are going to take care that every other vertex will have at least one outcoming edge in blue or green, $x_0$ will be fixed by the final coloring. Thus, after this step we consider $x_0$ as \emph{fixed}. 

In further steps we consider consecutive orbits, one by one, and level by level, starting from the largest orbit $O_1$ at the first level, and color all edges incident to the considered orbit (not colored yet) so that the coloring fixes all the vertices in this orbit. The latter is to be understood that for every coloring extending the coloring at the given step the automorphism group of the final colored graph fixes necessarily all the vertices in the orbit. The orbit considered at the given step and the level containing this orbit will be referred to as \emph{current}. 

After each step, we consider orbits of the subgroup  of $Aut(G,x_0)$ preserving colors of edges colored so far.  So in each step the orbits will become smaller and smaller. The set of orbits in a given step results from partitioning orbits considered in the preceding step. Coloring edges incident with the current orbits, we will have two aims in mind: one, to fix all the vertices in the orbit, and the second, to make the resulting new orbits small enough for further proceeding.

\subsection{Notation and terminology.} \label{s:not}To describe the procedure more precisely, and to prove it works, we introduce some terminology and notation. Given a current orbit $O$, by $H$ we denote  
the subgraph induced by the vertices of the orbit. Observe that since $O$ is assumed to be an orbit (under the action of some subgroup of $Aut(G,x_0)$), $H$ is transitive and all its connected components are isomorphic regular graphs. Moreover, all the vertices in $O$ have the same degree $d$ in $G$. 
Denoting by $m\geq 1$ the number of connected components of $H$, and by $n$ the cardinality of each such component, we get that the cardinality of the current orbit is $|O|=mn$.

We distinguish three types of the edges incident to $O$. The edges of the subgraph $H$ will be called the \emph{inner} edges of the orbit $O$. The remaining edges will be called \emph{incoming} (to $O$), if they are already colored, and \emph{outcoming} (not yet colored), otherwise. Each vertex in $O$ has the same number $t$ of incoming edges, $k$ of outcoming edges, and $r$ of inner incident edges ($r$ may be viewed also as the degree of the regular graph $H$). Of course, we have $t+r+k = d$. 

We process orbits one by one, and level by level, in the sense that as the next orbit to be processed we choose always one in the current level, if there is any orbit in this level not processed yet, or we choose any orbit in the next level, otherwise. As a result, if 
$O$ is a current orbit, then all orbits in earlier levels are already processed and perhaps some orbits in the current level, too.

Processing the current orbit $O$ we color all its inner edges  (if there are any) and all outcoming edges. As a result, at each step, all the edges having endpoints in earlier levels are colored, and therefore, since $G$ is connected:
\medskip

$(*)$ \textit{for the current orbit we have always $t>0$}.

\subsection{Conditions} \label{s:con}
Let $Aut^*(G,x_0)$ denote the subgroup  of $Aut(G,x_0)$ consisting of those automorphisms that preserve colors of the orbits colored so far. We are going to color the edges of the current orbit $O$ so that after coloring the following  conditions are satisfied:

\bigskip 

$(c1)$ \textit{All the vertices of the current orbit $O$  are fixed by~$A^*(G,x_0)$.} 

\medskip

$(c2)$ \textit{Each orbit $O$ of $A^*(G,x_0)$ that has an incoming colored edge either is trivial $($of size $1)$ or its size is not larger than $\delta-t+1$}. 

\medskip

$(c3)$ \textit{The only vertex whose all incident edges are colored red is $x_0$.}

\bigskip

Note that these conditions are satisfied after the initial step (for $(c2)$, this is so, since the orbits at level~1 have $t=1$). So, considering a current orbit $O$ we may assume that after every earlier step the conditions have been satisfied. In particular, they always hold for the current orbit $O$.

We note some simple consequences of $(c2)$ for future use. Since $\delta\leq t+r+k$, and in general, $n\geq r+1$, we have
\begin{equation}\label{e:m1}
 m \leq \frac{\delta-t+1}{r+1} \leq \frac{k+r+1}{r+1}  \leq k+1.
\end{equation}

In turn, combining directly $(c2)$ with $\delta \leq t+r+k$, we get that the size of $O$ satisfies also the inequality 
\begin{equation}\label{e:size}
 |O| \leq   r+k+1.
\end{equation} 
\smallskip

\textit{Remark}. Note that condition $(c1)$ follows from $(c2)$ and the fact that all edges of $O$ are colored (since in such a case we have $t\geq\delta$). Yet, we formulate $(c1)$ separately, as this is the first explicit aim of the processing each orbit.

\subsection{Single step.} \label{s:step} We describe now a single step in our procedure.
Let $O$ be a current orbit of cardinality $|O|\geq 1$ that has uncolored (inner or outcoming) edges. 
According to the observation~$(*)$ at the end of subsection~\ref{s:not} we may assume that $t>0$. For the time being, we assume also that $k>0$. (The way we treat orbits $O$ with $k=0$ will be described later in \ref{s:ter}.)

\subsubsection{Inner edges} We start from coloring the inner edges of $O$ to make all the connected components of $H$ rigid. At this point we assume that $n>2$. We apply induction on $\Delta$ assuming that our theorem holds for all graphs with the maximum degree less than $\Delta$. In particular, it holds for regular graphs with the vertex degree $r<\Delta$. It follows that (for $n> 2$) we may color the inner edges of $O$ so that each connected component has an asymmetric coloring. Now, it remains to fix just one vertex in each component, by suitable coloring of outcoming edges, to get all the vertices in $O$ fixed.

For $n=2$, the difference is that the connected components are $K_2$ and we have no way to make them rigid by coloring inner edges. So, we color the inner edges arbitrarily. For $n=1$ there is nothing to do at this point.

\subsubsection{Outcoming edges}\label{s:out}
Now, let us choose one vertex in each component. To fix attention, denote them $x_1,\ldots,x_m$ with $m\geq 1$. (Note that, in case when $n=1$, these are all the vertices in $O$). Let $U$ denote the set of those vertices outside $O$ that are endpoints of the outcoming (uncolored) edges of $O$. This set is nonempty (by assumption that $k>0$) and is partitioned into some orbits of $Aut^*(G,x_0)$, say, $Q_1,\ldots,Q_s$ with $s\geq 1$. 

From the properties of orbits, we know that each vertex $x_j\in O$ has the same number $k$ of outcoming edges, the edges from each $x_j$ go to each of the orbits $Q_i$, and the number of the edges going from $x_j$ to $Q_i$ is the same for every $x_j$; denote it by $k_i$. We have $k=k_1+k_2+\ldots+k_s$.  Moreover, for every pair $(Q_i,x_j)$ there is an edge going from $x_j$ to $Q_i$, which means, in particular, that all $k_i>0$ (the reader should be warned, that this does not mean that every vertex $x_j$ is adjacent to every vertex of $Q_i$.). The situation is illustrated in Figure~2. (In the picture, only the edges outcoming from $x_1$ are shown. The same number of edges partitioned in the same way comes out from each vertex $x_j$.)

\begin{figure}\label{fig2}
\begin{tikzpicture}[auto,inner sep=1pt, 
minimum size=2pt]

  
 \draw (7.2,2.9) node {$O$}; 

\draw (1.5,2.5) node[circle,fill=black,draw]{};  
\draw (1.85,2.55) node {$x_1$};

\draw (2.5,2.5) node[circle,fill=black,draw]{};  
\draw (2.85,2.55) node {$x_2$};

\draw (3.5,2.5) node[circle,fill=black,draw]{};  
\draw (3.85,2.55) node {$x_3$};

\draw (5,2.5) node {\Large $\dots$};

\draw (6,2.5) node[circle,fill=black,draw]{};  
\draw (6.35,2.55) node {$x_m$};

\draw  (4,2.5) ellipse (30mm and 4mm);

\draw[-] (1.5,2.5)--(0.7,0.5);
\draw[-] (1.5,2.5)--(0.9,0.5);
\draw[-] (1.5,2.5)--(1.1,0.5);
\draw[-] (1.5,2.5)--(1.3,0.5);

\draw[-] (1.5,2.5)--(2.5,0.5);
\draw[-] (1.5,2.5)--(2.7,0.5);

\draw[-] (1.5,2.5)--(3.9,0.5);
\draw[-] (1.5,2.5)--(4.1,0.5);

\draw[-] (1.5,2.5)--(5.3,0.5);

\draw[-] (1.5,2.5)--(6.5,0.5);

\draw[-] (1.5,2.5)--(7.2,0.5);
\draw[-] (1.5,2.5)--(7.6,0.5);

\draw (0.8,1.5) node {$k_1$};
\draw (1.75,1.5) node {$k_2$};
\draw (2.45,1.5) node {$k_3$};
\draw (5.5,1.5) node {$k_s$};

\draw  (1,0.5) ellipse (7mm and 2mm); 
\draw  (2.6,0.5) ellipse (6mm and 2mm); 
\draw  (4,0.5) ellipse (5mm and 2mm); 
\draw  (5.3,0.5) ellipse (5mm and 2mm); 
\draw  (6.5,0.5) ellipse (3mm and 1.5mm); 
\draw (7.2,0.5) node[circle,fill=black,draw]{};
\draw (7.6,0.5) node[circle,fill=black,draw]{};

\draw (1,0) node {$Q_1$};
\draw (2.6,0) node {$Q_2$};
\draw (4,0) node {$Q_3$};
\draw (6,0) node {\Large $\ldots$};
\draw (7.7,0) node {$Q_s$};


  \end{tikzpicture}
\caption{Edges outcoming form $x_1$ partitioned among orbits. 
}
\end{figure}
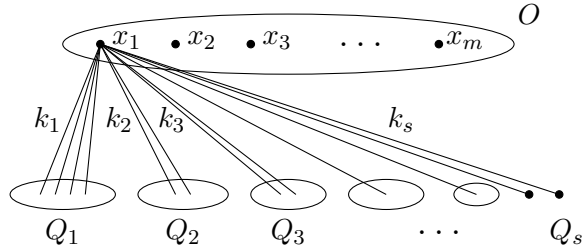

We wish to color edges outcoming from $x_1,\ldots,x_m$ so that each vertex $x_j$ gets a different uniform $(k_1,k_2,\ldots,k_s)$-palette 
$(p_1^j,p_2^j,\ldots,p_s^j)$ of colors. This is possible, since for $k_i>0$, by (\ref{e:m1}), $m \leq k+1$, and by Lemma~\ref{l:k+1(2)}, we have at least $k+1$ different uniform $(k_1,k_2,\ldots,k_s)$-palettes. 
It follows that after coloring the outcoming edges of the chosen vertices $x_1,\ldots,x_m$ as indicated above, we get all these vertices fixed.  

Moreover, for $n>2$, due to coloring of inner edges, we get all the vertices in $Q$ fixed. In this case, the edges outcoming from other vertices of $O$ are colored with the same palettes of colors as the edges outcoming from $x_1$. For $n=1$ there is nothing more to do.
For $n=2$,  we need to make use of the stronger part of inequality (\ref{e:m1}), that for $r=1$, $m \leq k/2+1 \leq k$. Then, we color the edges outcoming from the chosen vertices using $k$ different uniform $(k_1,k_2,\ldots,k_s)$-palettes, and the edges outcoming from the remaining vertices are colored using one more different uniform $(k_1,k_2,\ldots,k_s)$-palette. This guarantees that also in case $n=2$ we get all the vertices of $O$ fixed. (The reader should note that the whole procedure applies also if $|O|=1$.)

\subsubsection{Assigning palettes}\label{s:ap}
We make the above assignment of different palettes more precise with some additional conditions. First,  we  may postulate, and we do so, that $k_i$-palettes assigned to $x_1$, that is $p_1^1,p_2^1,\ldots,p_s^1$, involve no red color. Next, assuming that $k_1$ is the largest among all $k_i$, we require, in accordance with Lemma~\ref{l:k+1}, that $k_1$-palettes assigned to all $x_j$ (that is $p_1^j$) have the number of red colors $a\leq k_1/2$. This will be also the case for $k_i$-palettes with $i>1$.
For such palettes,  in accordance with the proof of Lemma~\ref{l:k+1(2)}, we assume that for each $i>1$ we fix two uniform $k_i$-palettes---one with no red color, and another one having at most one red color---and each $p_i^j$ is one of these two fixed palettes. If $k_i$ is odd, we may require that both the palettes involve no red color. In particular, for $k_i=1$, we assume that the corresponding two $1$-palettes are singletons consisting of blue or green.

\subsubsection{New orbits resulting from fixing $x_1,x_2,\ldots,x_m$.}\label{s:new}
If we assume that the vertices $x_i$ in $O$ are fixed, each orbit $Q_i$ is partitioned into smaller orbits of those vertices that have the same set of neighbors in $O$. Let $T=T(Q_i,S)$ be the subset of $Q_i$ whose neighbors in $O$ are exactly the vertices in a subset $S$ of $O$.  
Note that (in contrast with the situation in subsection~\ref{s:out}) the edges between $S$ and $T$ form a complete bipartite graph. Let $T_1,T_2,\ldots,T_w$, $w\geq 1$, denote all such nonempty sets $T$. They form a new finer partition of $U$ and we consider them as new orbits (of course, it may happen in an extreme case that they coincide with $Q_1,Q_2,\ldots,Q_s$.) 

Recall that for each vertex $x\in O$ we have assigned a uniform $k_i$-palette of colors for edges going from $x$ to $Q_i$. Using Lemma~\ref{l:part}, we partition this palette into uniform $k_j$-palettes corresponding to orbits $T_j$ contained in $Q_i$. Thus,  for each vertex $x\in O$ we have now assigned a uniform $k_j$-palette of colors for edges going from $x$ to $T_j$.  (Note that since this assignment is a detailing of the previous assignment, it keeps all the vertices in $O$ fixed.). What remains is to assign colors from the palette to particular edges, which may still lead to finer orbits. Our aim now is to show that we can do it so that the resulting orbits satisfy condition $(c2)$. 

\subsubsection{Coloring particular edges.}\label{s:coloring}
Let $T_j=T(Q_i,x_1^j,\ldots,x_u^j)$ be an orbit contained in $Q_i$ whose vertices are adjacent exactly to vertices $x_1^j,\ldots,x_u^j$ of~$O$. Let $t', r', k'$ denote, respectively, the number of edges incoming, inner, and outcoming from a vertex in $T_j$. The edges from $O$ to $T_j$ are considered as outcoming from $T_j$ (still uncolored). We will color them using assigned uniform palettes. Our aim is to show that as a result of this coloring, if $|T_j|>1$, we get orbit $T_j$ partitioned into new finer orbits such that each of them satisfies condition $(c2)$ (unless ). 

(i) First assume that $t'=0$. 
Then we start from $x_1^j$ and fix any coloring of the edges from $x_1^j$ to $T_j$ compatible with the uniform palette assigned to these edges. 
As a result, with this coloring, $T_j$ is partitioned into smaller orbits, each of them having size $s \leq \lceil k/2\rceil$ (since there are $k$ edges outcoming from $x_1^j$). As $t>0$, we have that $s \leq \lceil k/2\rceil \leq \lceil (\Delta-1)/2\rceil \leq \Delta/2 \leq \delta$,
which means that it satisfies ($c2$) (for $t'=1$ in place of~$t$). 

(ii) If $t'>0$ we apply a different strategy. 
First note that $t'>0$ means that some edge incoming to $T_j$ has been colored on some earlier step. Since we have assumed in \ref{s:con} that the conditions $(c1$--$c3)$ hold after every earlier step, in this case we may assume that $T_j$ satisfies $(c2)$ with $t=t'$. Again, we fix any coloring of the edges from $x_1^j$ to $T_j$ compatible with the uniform palette assigned to these edges, and as a result of this coloring, get a partition of $T_j$ into smaller orbits. Each of them has a size smaller than $|T_j|$, at least by one, which means that each of them satisfies $(c2)$ with $t=t'+1$, i.e., with $t$ greater by one. 

Now, in any of the above cases, we have $T_j$ partitioned into two or three smaller orbits (unless $|T_j|=1$). Then for each $x_i^j$, $i>1$, using Lemma~\ref{l:part}, we partition the uniform palette assigned to edges coming from $x_i^j$ to $T_j$ into uniform palettes corresponding to the sizes of the smaller orbits contained in $T_j$. 
Further, we repeat procedure (b) with each of the new orbits, getting still smaller orbits satisfying condition $(c2)$ and having assigned uniform palettes compatible with the first assignment in \ref{s:ap}. 

We proceed in this way with all vertices $x_i^j$ one by one, until all edges coming from $O$ to $T_j$ are colored. As a result $T_j$ is partitioned into orbits satisfying $(c2)$. 
It follows that after coloring all the edges outcoming from $O$ all the resulting orbits contained in $U$ satisfy condition $(c2)$, as required.

\subsubsection{Terminal orbits} \label{s:ter} It may happen that some of the orbits $Q_1,\ldots,Q_s$, after coloring all the edges incoming from $O$, have no outcoming edges. Then we have additional tasks: (a) to fix all the vertices in the orbit $Q_i$ with this property, and (b) to make sure that no vertex in $Q_i$ gets all incident edges colored red.

We first look at the final partition of $U$ into orbits after coloring all edges incoming to $Q_i$ from $O$. If $T'\subseteq Q_i$ is a member of this partition and it has no outcoming edges, then the number $k'$ of the edges outcoming from $T'$ equals $0$. Denote by $t'$ and $r'$ the number of incoming and inner edges of $T'$, respectively.  Then,  by $(\ref{e:size})$, $|T'| \leq r'+1$ .  It follows that the graph $H'$ induced by the vertices of $T'$ is the complete graph $K_{r'+1}$ on $r'+1$ vertices. 

If $r'>1$, then as we have remarked at the beginning of Section~\ref{s:pre}, $K_{r'+1}$ has an asymmetric 3-coloring in which each vertex has an incident edge colored blue or green. This completes both the tasks (a) and (b) for the vertices of $T'$. 

If $r'=1$ then $T'=K_2$ consists of a single edge. Coloring this edge in blue or green completes task (b), but we have no way to guarantee task (a) in this case. Therefore, we go back to step~\ref{s:coloring}, and add an additional condition when coloring edges in this special case. 

Suppose that $T'\subseteq T_j$, and let $xy$ be the only edge of $T'$. Since there are no outcoming (uncolored) edges from $T'=K_2$, it means that
none of $x$ and $y$  has outcoming uncolored edges in $T_j$ other than those going to $O$. It follows, that $T_j$ as an orbit consists of one or more copies of $K_2$, and has 
no outcoming uncolored edges other than those going to $O$.

In~\ref{s:coloring} we have colored edges going from $x_1^j$ to $T_j$ in an arbitrary way compatible with the uniform palette assigned to it. We may add one condition that two edges going to vertices joined by an inner edge in $T_j$ do not get the same color. This is possible since the palette is uniform. This guarantees that the case $r'=1$ cannot occur at this stage, as the vertices joined by an edge in $T_j$ go to different final orbits.

\subsubsection{Terminal orbits with $r'=0$} \label{s:ter0} The case when $r'=0$ in the preceding subsection is more complicated. Then $T'=K_{r'+1}$ consists of a single vertex $z$ and to make sure that there is a blue or green edge incident to $z$, we need to come back to subsection \ref{s:new}, and treat the whole orbit $Q_i$ containing $z$ in a more careful way. 

Similarly as in the previous subsection for $r'=1$ we infer that no vertex in $Q_i$ has outcoming edges, and there are no inner edges in $Q_i$. 

First assume, in addition, that $Q_i$ has incoming colored edges (before coloring the edges incoming from $O$). Then, as in case (ii) of \ref{s:coloring}, it satisfies condition~$(c2)$. By (\ref{e:size}), $|Q_i|\leq k'+1$, where $k'$ is the number of uncolored edges outcoming from a vertex of $Q_i$. All these edges have endpoints in $O$. On the other hand, each vertex $x_j\in O$ has a neighbor in $Q_i$, and the edges going from $x_j$ to $Q_i$ are to be colored with the assigned uniform $k_i$-palette. 

If $k_i=1$ then, according to~\ref{s:ap}, the corresponding edges are colored blue or green, which completes task (a) of \ref{s:ter}. For (b), note that in this case all the vertices in $Q_i$ have different sets of neighbors in $O$. Hence, all they are fixed, once the vertices of $O$ are fixed. 

Thus, we may assume, that $k_i>1$. We describe the way of assigning colors to all edges from $O$ to $Q_i$, with accordance to assigned $k_i$-palettes, so that each vertex of $Q_i$ gets an edge colored blue or green (to complete task (b) of~\ref{s:ter}).
We start from edges going from $x_1$ to $Q_i$. By the conditions assumed in subsection~\ref{s:ap}, all edges going from $x_1$ are not red, so the endpoints of these edges in $Q_i$ we may treat as \emph{settled} (from the point of view of condition (b)). The vertex $x_1\in O$ we treat now as \emph{used} (all the edges going from $x_1$ to $Q_i$ are colored). Coloring the edges of $x_1$ partition $Q_i$ into smaller orbits, each of which satisfies condition $(c2)$ for $t$ larger, at least, by one. There are two such orbits of the vertices adjacent to $x_1$, and one more consisting of the remaining vertices in $Q_i$, not settled yet. 

If the third set is nonempty, in the next step, we choose a vertex $y$ belonging to this set. It must be adjacent to some vertex $x_j$ other than $x_1$. We have a uniform $k_i$-palette assigned to $x_j$ and $k_i$ edges going from $x_j$ to $Q_i$ not colored yet. We partition this palette, in accordance with Lemma~\ref{l:part},  among the three new orbits in $Q_i$ that resulted from coloring the edges from $x_1$ to $Q_i$.  
Thus, each of these orbit has assigned a uniform palette. Because the palette assigned to $x_j$ and $O_i$ has $a\leq k_i/2$ (according to \ref{s:ap}), we can partition it so that there is a blue or green color assigned to the orbit containing $y$. So, we may put this color to the edge $yx_j$, which makes $y$ settled. Other edges from $x_j$ to $Q_i$ are colored, in accordance to the uniform palettes assigned to the three orbits, which yields a new finer partition of $Q_i$ into orbits. Because the palettes are uniform, all the resulting new orbits, with new colored edges incoming, satisfy condition $(c2)$  with $t$ increased by one. The argument is, generally, the same as in~\ref{s:new} (ii). A difference is that there may be only one edge incoming to some orbit from $x_j$, while other vertices in this orbit may not be adjacent to $x_j$, at all. Yet, this also partition the orbit in question into two new orbits, both satisfying condition~$(c2)$.

We proceed in the same way with further unsettled vertices $y\in Q_i$, partitioning uniform palettes, and obtaining a set of new finer orbits satisfying~$(c2)$. The situation is illustrated in Figure~3: the first step on the left, and a further step on the right. Note that there may be edges going from the used vertices to the unsettled that are colored red. They do not matter in the procedure.

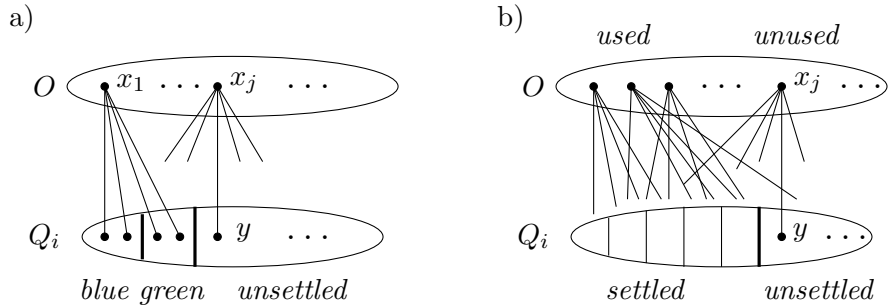
\begin{figure}\label{fig3}
\begin{tikzpicture}[auto,inner sep=1pt, 
minimum size=2pt]


 \draw (0.2,3.4) node {a) }; 
  \draw (0.5,2.5) node {$O$}; 
   \draw (0.5,0.5) node {$Q_i$}; 
 \draw  (3,2.5) ellipse (22mm and 4mm); 
  \draw  (3,0.5) ellipse (20mm and 4mm); 
 
\draw (6.7,3.4) node {b) }; 
\draw (8.2,3.2) node {\textit{used} }; 
\draw (10.5,3.2) node {\textit{unused}}; 

  \draw (7,2.5) node {$O$}; 
   \draw (7,0.5) node {$Q_i$}; 
\draw  (9.5,2.5) ellipse (22mm and 4mm); 
\draw  (9.5,0.5) ellipse (20mm and 4mm);


\draw (1.3,2.5) node[circle,fill=black,draw]{}; \draw (1.65,2.55) node {$x_1$};

\draw (2.8,2.5) node[circle,fill=black,draw]{};  \draw (3.15,2.55) node {$x_j$};
 \draw (4,2.5) node {\Large$\ldots$};
  \draw (2.3,2.5) node {\Large$\ldots$};


\draw (7.8,2.5) node[circle,fill=black,draw]{}; 
\draw (8.3,2.5) node[circle,fill=black,draw]{};
\draw (8.8,2.5) node[circle,fill=black,draw]{};

\draw (10.3,2.5) node[circle,fill=black,draw]{};  \draw (10.65,2.55) node {$x_j$};
 \draw (11.35,2.5) node {\Large$\ldots$};
 \draw (9.5,2.5) node {\Large$\ldots$};


\draw (1.3,0.5) node[circle,fill=black,draw]{};
\draw (1.6,0.5) node[circle,fill=black,draw]{};
\draw (2,0.5) node[circle,fill=black,draw]{};
\draw (2.3,0.5) node[circle,fill=black,draw]{};
\draw (2.8,0.5)
node[circle,fill=black,draw]{};
 \draw (3.15,0.55) node {$y$};
 \draw (4,0.5) node {\Large$\ldots$};

\draw[-,very thick] (1.8,0.8)--(1.8,0.2);
\draw[-,very thick] (2.5,0.9)--(2.5,0.1);

\draw[-] (1.3,2.5)--(1.3,0.5);
\draw[-] (1.3,2.5)--(1.6,0.5);
\draw[-] (1.3,2.5)--(2,0.5);
\draw[-] (1.3,2.5)--(2.3,0.5);

\draw[-] (2.8,2.5)--(2.8,0.5);

\draw[-] (2.8,2.5)--(2.1,1.5);
\draw[-] (2.8,2.5)--(2.4,1.5);
\draw[-] (2.8,2.5)--(3.1,1.5);
\draw[-] (2.8,2.5)--(3.4,1.5);


\draw (10.3,0.5)
node[circle,fill=black,draw]{};
 \draw (10.55,0.55) node {$y$};
 \draw (11.15,0.5) node {\Large$\ldots$};

\draw[-] (8,0.75)--(8,0.25);
\draw[-] (8.5,0.85)--(8.5,0.15);
\draw[-] (9,0.9)--(9,0.1);
\draw[-] (9.5,0.9)--(9.5,0.1);
\draw[-,very thick] (10,0.9)--(10,0.1);

\draw[-] (7.8,2.5)--(7.8,0.8);
\draw[-] (7.8,2.5)--(8.1,0.9);
\draw[-] (7.8,2.5)--(8.4,1);
\draw[-] (7.8,2.5)--(8.7,1);


\draw[-] (8.3,2.5)--(8.25,1.);
\draw[-] (8.3,2.5)--(9.1,1);
\draw[-] (8.3,2.5)--(9.4,1);
\draw[-] (8.3,2.5)--(9.7,1);
\draw[-] (8.3,2.5)--(10.5,1);

\draw[-] (8.8,2.5)--(8.5,1.);
\draw[-] (8.8,2.5)--(8.8,1);
\draw[-] (8.8,2.5)--(9.3,1);
\draw[-] (8.8,2.5)--(9.8,1);

\draw[-] (10.3,2.5)--(10.3,0.5);

\draw[-] (10.3,2.5)--(9,1.2);
\draw[-] (10.3,2.5)--(9.7,1.4);
\draw[-] (10.3,2.5)--(10,1.5);
\draw[-] (10.3,2.5)--(10.6,1.5);

\draw (10.8,-0.2) node {\textit{unsettled}};
\draw (8.5,-0.2) node {\textit{settled}};
\draw (1.3,-0.2) node {\textit{blue}};
\draw (3.8,-0.2) node {\textit{unsettled}};

\draw (2.2,-0.28) node {\textit{green}};

  \end{tikzpicture}
\caption{Procedure of coloring edges in the case of terminal orbits $Q_i$ without incoming edges.
}
\end{figure}

We do it
until all the vertices in $Q_i$ are settled.  Since $|Q_i|\leq k'+1$ and there are $k'$ edges coming to $O$ from each vertex of $Q_i$, there are enough vertices in $O$ to settle all the vertices in $Q_i$ (recall that at the beginning, using $x_1$, we have settled at least two vertices simultaneously).  For the remaining vertices $x_j\in O$ not used in this procedure (if there are any), we color the edges outcoming from them to $Q_i$, according to assigned uniform palettes, partitioned among orbits in accordance with Lemma~\ref{l:part}, so that all the resulting orbits always satisfy $(c2)$. Since for each orbit at the end we have $k=r=0$, it follows by (\ref{e:size}), that each such orbit consists of a single fixed vertex, which makes task (a) completed, as well.
\smallskip

It remains to consider the case when $Q_i$ has no incoming colored edges. Then $k'\geq \delta$. On the other hand, by $(c2)$ and $t>0$, we have $|O|\leq \delta$. It follows that in this case $|O|=\delta$ and the edges between $Q_i$ and $O$ form a complete bipartite graph. If so, then the edges going to $Q_i$ from $x_1$ alone, according to assumption in subsection~\ref{s:ap}, make that each vertex in $Q_i$ has an incident edge in blue or green, completing task (b). Task (a) is completed by coloring the edges in the same way as in~\ref{s:coloring}. As before, the fact that condition $(c2)$ is satisfied means, in case $k=r=0$, that all the vertices in $Q_i$ are fixed.

\subsubsection{Terminal orbits at level~1} \label{s:ter1} It may happen that a terminal orbit $O$ occurs at level~1, which is a case not covered by the previous subsection.  Then, the degree of vertices in $O$ is equal to $d=1+r \geq \delta$, and it follows that $H=K_\delta$ is the complete graph on $\delta$ vertices and there is no other orbit at level~1. Since $G$ is connected, it follows that $G=K_{\delta+1}$, and we know that the theorem holds in this case.

\subsection{Proof of the theorem}
We show that using the procedure described above we get an asymmetric 3-coloring of $G$. 

First, observe that since at each step a new orbit is processed with a least one new vertex, and this is done level by level, we make sure that each vertex of $G$ is processed within some orbit at some step.

Next, we show that, during the procedure, for each orbit $O$ after processing it as described in~\ref{s:step}, the conditions $(c1$--$c3)$ are satisfied. Indeed, as observed in~\ref{s:con}, the conditions are satisfied after initial step of coloring edges incident to $x_0$ in red. At each step, by \ref{s:not}$(*)$, the next orbit to be processed  has $t>0$. If in addition, $k>0$, then the conditions $(c1$--$c3)$ are satisfied by what established in~\ref{s:step}. If $k=0$, then $O$ is a terminal orbit. If it occurs at level~1, then $G$ has an asymmetric 3-coloring by~\ref{s:ter1}. Otherwise, $O$
has been processed at some earlier step as described in~\ref{s:ter} and~\ref{s:ter0}, and there is no need of further processing. According to this, also in this case the conditions $(c1)$ and $(c3)$ are made sure to be satisfied, while $(c2)$ is satisfied trivially (due to earlier steps). 

Thus, when the procedure is completed, all the edges are colored, and all the vertices are fixed by the subgroup of $A(G,x_0)$ preserving colors of the edges. Since by $(c3)$ no vertex has the same coloring of incident edges as $x_0$, it follows that all the vertices are fixed by the automorphism group of the resulting edge-colored graph, which completes the proof.

\subsection{Conclusion} Taking into account conjectures discussed in \cite{LPS} and \cite{IKP}, examples of small graphs with $D'(G)=3$ given after Theorem~1 seem to be exceptional. The authors of \cite{LPS} conjecture that, actually, for all the remaining regular graphs $D'(G)=2$. We believe that the situation is similar in the case of graphs satisfying the conditions of our theorem. Yet, proving such a result seem to require completely different tools.


\begin{thebibliography}{99}


\bibitem{AC} M. Albertson and K. Collins, Symmetry breaking in graphs, Electron. J. Combin. 3 (1996) R18.

\bibitem{AS}
S. Alikhani and S. Soltani, The distinguishing number and the distinguishing index of line and graphoidal graph(s), AKCE Int. J. Graphs and Combin. 17(1) (2020) 1-6.


\bibitem{ba} L. Babai, Asymmetric trees with two prescribed valences, Acta Math. Acad. Sci. Hung. 29 (1977) 193-200.


\bibitem{BP}
I. Broere, M. Pil\'{s}niak, The distinguishing index of the Cartesian product of countable graphs, Ars Math. Contemp. 13 (2017) 15-21.



\bibitem{FI}
M.J. Fisher, G. Isaak, Distinguishing colorings of Cartesian products of complete graphs, Discrete Math. 308 (2008) 2240-2246.


\bibitem{GK} M. Grech, A. Kisielewicz {Distinguishing index of graphs with simple automorphism groups}, 2020, submitted to European J. Combin.

\bibitem{IKP}
W. Imrich, R. Kalinowski, M. Pil\'{s}niak, M. Wo\'{z}niak, The distinguishing index of connected graphs without pendant edges, Ars Math. Contemp. 18 (2020) 117-126.


\bibitem{KP} R. Kalinowski, M. Pil\'{s}niak, Distinguishing graphs by edge-colourings, European J. Combin. 45 (2015) 124-131.



\bibitem{le} F. Lehner, Breaking graph symmetries by edge colourings, J. Combin. Theory, Ser. B, 127 (2017) 205-214.



\bibitem{LPS} F. Lehner, M. Pil\'{s}niak, M. Stawiski,  A bound for the distinguishing index of regular graphs, European J. Combin. 89 (2020) 103145.

\bibitem{pi}
M. Pil\'{s}niak, Improving upper bounds for the distinguishing index, Ars Math. Contemp. 13 (2017) 259-274.

\bibitem{pi2}
M. Pil\'{s}niak, Nordhaus-Gaddum type inequalities for the distinguishing index, Ars Math. Contemp. (2021), accepted, doi:10.26493/1855-3974.2173.71a

\bibitem{PS}
M. Pil\'{s}niak, M. Stawiski, The optimal general upper bound for the distinguishing index of infinite graphs, J. Graph
Theory 93 (2020) 463–469.

\bibitem{PT} M. Pil\'{s}niak, T. Tucker, Distinguishing index of maps, European J. Combin. 84 (2020) 103034.

 
\end{thebibliography}
\end{document}